\documentclass[12 pt]{article}%
\usepackage{amsmath, amsfonts, amsthm, color,latexsym}
\usepackage{amsmath, ulem}
\usepackage{amsfonts}
\usepackage{amssymb}
\usepackage{color, soul}
\usepackage[all]{xy}
\usepackage{graphicx}%
\providecommand{\U}[1]{\protect\rule{.1in}{.1in}}
\allowdisplaybreaks[4]
\newtheorem{theorem}{Theorem}[section]
\newtheorem{proposition}[theorem]{Proposition}
\newtheorem{corollary}[theorem]{Corollary}
\newtheorem{example}[theorem]{Example}

\newtheorem{remark}[theorem]{Remark}

\newtheorem{lemma}[theorem]{Lemma}
\newtheorem{final remark}[theorem]{Final Remark}

\allowdisplaybreaks[4]

\newcommand{\cvfa}{\overset{|\sigma|(E,E^{*})}{\longrightarrow}}
\newcommand{\cvfea}{\overset{|\sigma|(E^{*}, E)}{\longrightarrow}}
\newcommand{\cvfdea}{\overset{|\sigma|(E^{**}, E^{*})}{\longrightarrow}}
\newcommand{\cvf}{\overset{\omega}{\rightarrow}}

\newcommand {\cvfe} {\overset{\omega^\ast}{\rightarrow}}
\newcommand {\R}{\mathbb{R}}

\newcommand {\N} {\mathbb{N}}

\begin{document}

\title{Grothendieck's compactness principle for the absolute weak topology}
\author{Geraldo Botelho\thanks{Supported by Fapemig Grant PPM-00450-17. }\, , Jos\'e Lucas P. Luiz and  Vinícius C. C. Miranda\thanks{Supported by CNPq Grant 150894/2022-8\newline 2020 Mathematics Subject Classification: 46B42, 46B40, 46A50.\newline Keywords: Banach lattices, absolute weak compactness, absolute weak sequential compactness, positive Schur property. }}
\date{}
\maketitle

\begin{abstract} We prove the following results: (i) Every absolutely weakly compact set in a Banach lattice is absolutely weakly sequentially compact. (ii) The converse of (i) holds if $E$ is separable or $B_{E^{**}}$ is absolutely weak$^*$ compact. (iii) Every absolutely weakly compact subset of a Banach lattice is contained in the closed convex hull of an absolutely weakly null sequence if and only if the Banach lattice has the positive Schur property. Examples and applications are provided.
\end{abstract}

\section{Introduction and background}
The compactness principle proved by Grothendieck in \cite[p.112]{grothendieck} states that every norm compact subset of a
Banach space is contained in the closed convex hull of a norm null sequence. In \cite{dowling}, the authors studied a version of this principle concerning the weak topology and proved the following outstanding result: every weakly compact subset of a Banach space is contained in the closed convex hull of a weakly null sequence if and only if the Banach space has the Schur property (meaning that weakly null sequences are norm null). The hard part of their proof uses basic sequence techniques. Later, another proof appeared in \cite{johnson} using an operator theoretic approach. As expected, Grothendieck compactness-type principles have been considered for different topologies. For instance, P. N. Dowling and D. Mupasiri \cite{mupasiri} proved a Grothendieck compactness principle for the Mackey dual topology. 
We also refer the reader to \cite{causey} and \cite{dowfree} for different perspectives concerning Grothendieck's compactness principle.

In Banach lattice theory, topologies related to the order structure play a central role. Perhaps the most known of these topologies is the absolute weak topology $|\sigma|(E, E^*)$ on a Banach lattice $E$, which lies between the weak and norm topologies. It is known that the weak and the absolute weak topologies coincide only on finite dimensional Banach lattices. Compact (and sequentially compact) sets in these topologies are also different in general (cf. Example \ref{priexe}). The main purpose of this paper is to prove a Grothendieck compactness principle for the weak absolute topology. The lattice counterpart of the Schur property is the positive Schur property (meaning that positive weakly null sequences are norm null), so, bearing in mind the weak compactness principle proved in \cite[Theorem 1]{dowling}, if the weak topology is replaced with the absolute weak topology, then it is natural to replace the Schur property with the positive Schur property. So, from the beginning we tried to prove that every absolutely weakly compact subset of a Banach lattice is contained in the closed convex hull of an absolutely weakly null sequence if and only if the Banach lattice has the positive Schur property.  

In our way to prove the main result, we realized that some well known results on the weak topology on Banach spaces that are often used in the proofs of compactness principles have no known analogues for the absolute weak topology on Banach lattices. The result is that we were forced to prove lattice versions for the absolute weak topology of known results for the weak topology. As any expert on the field knows, the two implications of the Eberlein-Šmulian Theorem (weak compactness coincides with weak sequential compactness in Banach spaces) are among these results. Eberlein-Šmulian-type results have been studied for different topologies in Banach lattices, a very recent development can be found in \cite[Section 8]{kandic}. But we have not been able to find anything in this direction for the absolute weak topology. In Section 2 we prove that a Šmulian Theorem holds in full generality: absolutely weakly compact subsets of Banach lattices are absolutely weakly sequentially compact (cf. Theorem \ref{smulian}). Since Eberlein's Theorem is usually proved applying Alaoglu's Theorem, we investigated the validity of an analogue of Alaoglu's Theorem for the absolute weak topology. We ended up proving that this is not case (cf. Proposition \ref{noalaoglu} and Example \ref{ex77}). We thus proved a partial Eberlein Theorem, establishing that absolutely weakly sequentially compact subsets of Banach lattice $E$ are absolutely weakly compact if $E$ is separable or the closed unit ball of $E^{**}$ is absolutely weak$^*$-compact (cf. Theorem \ref{eberlein}). Examples of nonseparable Banach lattices for which the closed unit ball of $E^{**}$ is absolutely weak$^*$-compact are provided.

Although Theorem \ref{eberlein} is only a partial version of Eberlein's Theorem, in Section 3 we managed to overcome this drawback. In the end we prove the Grothendieck compactness principle for the absolute weak topology exactly as we have tried from the beginning (cf. Theorem \ref{mainteo}). It is worth mentioning that a crucial part of the proof of \cite[Theorem 1]{dowling}, which we isolated in Lemma \ref{lema}, plays an important role in our proof as well. As an application we obtain two characterizations of the dual positive Schur property relating it to the absolute weak$^*$ topology.  




We refer the reader to \cite{alip, meyer} for background on Banach lattices and to \cite{fabian} for Banach space theory. Throughout this paper $X, Y$ denote Banach spaces and $E, F$ denote Banach lattices. We denote by $B_X$ the closed unit ball of $X$ and by $X^*$ its topological dual. For a subset $A \subset X$, $\overline{\text{co}}(A)$ denotes the closed convex hull of $A$. As usual, the weak topology on a Banach space $X$ shall be denoted by $\sigma(X,X^{*})$ or $\omega$, and the weak$^*$ topology on $X^*$ by $\sigma(X^{*},X)$ or $\omega^*$.


\section{Eberlein-Šmulian Theorem}
For a Banach lattice $E$ and a nonempty subset $A$ of the topological dual $E^*$ of $E$, the absolute weak topology $|\sigma|(E,A)$ is the locally convex-solid topology on $E$ generated by the family $\{p_{x^{*}}: x^{*} \in A\}$ of lattice seminorms, where $p_{x^{*}}(x) = |x^{*}| (|x|)$ for all $x \in E$ and $x^{*} \in A$. For all $\varepsilon > 0$ and $x_1^{*}, \dots, x_n^{*} \in A$, set
$$ V(0;\, |x_1^{*}|, \dots, |x_n^{*}|;\, \varepsilon) := \{x \in E : |x_i^{*}|(|x|) < \varepsilon, i=1,\ldots,n\}. $$
The collection of all sets of this form is a basis of zero neighborhoods for $|\sigma|(E, A)$. Morever, the weak topology $\sigma(E,A)$ on $E$ determined by the functionals belonging to $A$ is contained in $|\sigma|(E, A)$. In this paper, we are interested in the case $A = E^*$.
It follows from \cite[Theorems 2.36 and 2.38]{alip2} that the weak topology $\sigma(E,E^*)$ and the absolute weak topology $|\sigma|(E,E^*)$ coincide only on finite dimensional Banach lattices. Naturally, this remarkable fact raises the question of whether analogues of known results for the weak topology in Banach spaces hold for the absolute weak topology in Banach lattices. In this section we address this question for the Eberlein-Šmulian Theorem. 

We start by showing that the absolute weak topology on a Banach sublattice coincides with the corresponding induced topology. By $\tau |_Y$ we denote the topology on a subset $Y$ of a topological $X$ induced by the topology $\tau$ of $X$.

\begin{lemma} \label{induzida}
For a Banach sublattice $F$ of the Banach lattice $E$, the absolute weak topology on $F$ coincides with the topology induced by the absolute weak topology of $E$, that is, $(F, |\sigma|(F, F^{*})) = (F, |\sigma|(E,E^{*})|_{F})$.
\end{lemma}

\begin{proof} It is enough to check that basic neighborhoods of zero in one of the topologies belong to the other one. In one direction, apply the Hahn--Banach-type extension theorem for positive functionals (see \cite[Corollary 1.3]{lotz}); and in the other direction just use that the restriction of a positive functional to a sublattice is a positive functional. 
%
\end{proof}

As usual, we say that a net $(x_\alpha)_\alpha$ in a Banach lattice $E$ is absolutely weakly convergent if it is convergent with respect to the absolute weak topology, that is, if there exists $x \in E$ such that $x_\alpha \cvfa x$. In particular, $(x_\alpha)_\alpha$ is said to be absolutely weakly null if $x_\alpha \cvfa 0$. The following facts shall be used without reference.

\begin{remark}\label{remconv}\rm (i) $x_\alpha \cvfa x$ in $E$ if and only if $x^{*}(|x_\alpha - x|) \longrightarrow 0$ for every $0 \leq x^{*} \in E^{*}$. Thus, every absolutely weakly convergent sequence is weakly convergent to the same element.\\
(ii) A net $(x_\alpha)_\alpha$ in a Banach lattice is absolutely weakly null if and only if $(|x_\alpha|)_\alpha$ is weakly null. In particular, a positive net is absolutely weakly null in a Banach lattice $E$ if and only if it is weakly null.\\
(iii) A disjoint sequence is absolutely weakly null in a Banach lattice if and only if it is weakly null.

(i) and (ii) follow from the definition of the absolute weak topology. For (iii), if $(x_n)_n$ is a disjoint weakly null sequence in a Banach lattice $E$, then $|x_n| \stackrel{\omega}{\longrightarrow} 0$ by \cite[Proposition 1.3]{wnukdual}, hence $(x_n)_n$ is absolutely weakly null.
\end{remark}

  A subset $K$ of a Banach lattice $E$ is said to be absolutely weakly compact (respectively, absolutely weakly sequentially compact) if it is compact in the topology $|\sigma|(E, E^{*})$, or simply, if it is $|\sigma|(E, E^{*})$-compact (respectively. if it is $|\sigma|(E, E^{*})$-sequentially compact, that is, if every sequence in $K$ has an absolutely weakly convergent subsequence to an element of $K$). It is immediate that every absolutely weakly compact subset of a Banach lattice is weakly compact. Moreover, as a consequence of the Eberlein-Šmulian Theorem every absolutely weakly sequentially compact subset of a Banach lattice is weakly compact as well.

\begin{lemma} \label{absweak2}
If a subset $K$ of a separable Banach lattice $E$ is absolutely weakly compact or absolutely weakly sequentially compact, then the absolute weak topology $|\sigma|(E, E^{*})$ is metrizable on $K$.
\end{lemma}

\begin{proof} Let $\{x_n : n \in \N \}$ be a dense subset of $E$. By Hahn-Banach we can take $D = \{x_n^{*} :n \in \N \} \subset E^{*}$ such that $\|x_n^{*}\| = 1$ and $x_n^{*}(x_n) = \|x_n\|$ for every $n \in \N$. Combining the denseness of the set $\{x_n: n \in \N \}$ with these properties of the functionals $(x_n^{*})_n$, it is easy to check that the bounded linear operator
$$T \colon E \longrightarrow \ell_\infty~,~T(x) = (x_n^{*}(x))_n,$$ is an isometric embedding. Consider the locally convex-solid topology $|\sigma|(E, D)$ which is generated by the sequence $\{p_{x_n^{*}}: n \in \N \}$ of seminorms given by $p_{x_n^{*}}(x) = |x_n^{*}|(|x|)$ for every $x \in E$. If $x \in E$ is such that $p_{x_n^{*}}(x) = 0$ for every $n \in \N$, then
$$ |x_n^{*}(x)| \leq |x_n^{*}|(|x|) = p_{x_n^{*}}(x) = 0 $$
for every $n \in \N$, therefore $\|x\| = \|T(x)\| = 0 $. By \cite[Lemma 5.76]{aliprantis} it follows that $|\sigma|(E, D)$ is a Hausdorff locally convex topology on $E$ generated by the sequence $\{p_{x_n^{*}}: n \in \N \}$ of seminorms. So, by \cite[Lemma 5.75]{aliprantis} we conclude that $|\sigma|(E, D)$ is metrizable. Moreover, since $D \subset E^{*}$, it is immediate that $|\sigma|(E, D) \subset |\sigma|(E, E^{*})$, hence the identity map $Id \colon (K, |\sigma|(E, E^{*})) \longrightarrow (K, |\sigma|(E, D))$ is continuous.

Supposing that $K$ is absolutely weakly compact, we get from \cite[Theorem 2.36]{aliprantis} that $Id$ is an homeomorphism because  $(K, |\sigma|(E, E^{*}))$ is compact and $(K, |\sigma|(E, D))$ is Hausdorff. Therefore $|\sigma|(E, D)$ and $|\sigma|(E, E^{*})$ coincide on $K$, from which it follows that $(K, |\sigma|(E, E^{*}))$ is metrizable.

Assume now that $K$ is absolutely weakly sequentially compact. In order to prove that $Id$ is an homeomorphism, let $F$ be a $|\sigma|(E, E^{*})|_{K}$-closed subset of $K$ and let $(x_n)_n \subset F$ be such that $x_n \cvfa x$ for some $x \in K$. Since $F \subset K$, the assumption yields that  there exist a subsequence $(x_{n_k})_k$ of $(x_n)_n$ and $y \in K$ such that $x_{n_k} \cvfa y$. As $F$ is $|\sigma|(E, E^{*})$-closed, $y \in F$. Moreover, since $D \subset E^{*}$, $x_{n_k} \overset{|\sigma|(E,D)}{\longrightarrow} y$, which gives $x = y$.
    This proves that $Id(F) = F$ is sequentially $|\sigma|(E,D)$-closed and, since $|\sigma|(E,D)$ is metrizable, we obtain that $F$ is $|\sigma|(E,D)$-closed. Thus $Id$ is an homeomorphism. It follows that $|\sigma|(E, D)|_K$ and $|\sigma|(E, E^{*})|_K$ coincide on $K$, and we conclude that $(K, |\sigma|(E, E^{*}))$ is metrizable.
\end{proof}

We are ready to prove the Šmulian part of the Eberlein-Šmulian Theorem for the absolute weak topology.

\begin{theorem} \label{smulian}
    Absolutely weakly compact subsets of Banach lattices are absolutely weakly sequentially compact.
\end{theorem}

\begin{proof} Let $K$ be an absolutely weakly compact subset of a Banach lattice  $E$. Given a sequence $(x_n)_n$ in $K$, put $X = \overline{ \text{span}} \{x_n: n \in \N\}$. It follows from \cite[Exercise 9, p.\,204]{alip} that there exists a separable Banach sublattice $F$ of $E$ containing $X$. Since the absolute weak topology $|\sigma|(F,F^{*})$ on $F$ coincides with the absolute weak topology $|\sigma|(E,E^{*})$ of $E$ induced on $F$ (Proposition \ref{induzida}), we have that $K \cap F$ is $|\sigma|(F, F^{*})$-compact, so $(K \cap F, |\sigma|(F, F^{*}))$ is metrizable by Lemma \ref{absweak2}. As $(x_n)_n$ is contained in $K \cap F$, there exist a subsequence $(x_{n_k})_k$ of $(x_n)_k$ and an element $x \in K \cap F$ such that $x_{n_k} \longrightarrow x$ with respect to the topology $|\sigma|(F,F^{*})$. Calling on Proposition \ref{induzida} once again we get the convergence $x_{n_k} \cvfa x$, which completes the proof.
\end{proof}

As a first application we give an example of a weakly compact set which is neither absolutely weakly compact nor absolutely weakly sequentially compact.

\begin{example}\label{priexe} \rm It is well known that the sequence $(r_n)_n$ formed by the Rademacher functions is weakly null in $L_1[0,1]$. So, $K:= \{r_n : n \in \mathbb{N}\} \cup \{0\}$ is a weakly compact subset of $L_1[0,1]$. Suppose that $K$ is absolutely weakly sequentially compact. In this case there are a subsequence $(r_{n_j})_j$ of $(r_n)_n$ and $f \in L_1[0,1]$ such that $r_{n_j} \cvfa f$. Since the weak topology is contained in $|\sigma|(E,E^{*})$, the subsequence $(r_{n_j})_j$ is weakly convergent to $f$. But $(r_{n_j})_j$ is weakly null as a subsequence of a weakly null sequence, so $f = 0$, meaning that $r_{n_j} \cvfa 0$. By Remark \ref{remconv}(ii) it follows that $(|r_{n_j}|)_j$ is weakly null. But, for every $j \in \mathbb{N}$, $|r_{n_j}|$ is the constant function equal to $1$, so $(|r_{n_j}|)_j$ converges weakly to the constant function equal to $1$. This contradiction shows that there is no $|\sigma|(E,E^{*})$-convergent subsequence of $(r_n)_n$, proving that $K$ is not absolutely weakly sequentially compact. Theorem \ref{smulian} assures that $K$ is not absolutely weakly compact.
\end{example}

The usual proof of the Eberlein part of the Eberlein-Šmulian Theorem uses the  weak$^*$ compactness of the closed unit ball of the dual of any Banach space (Alaoglu's Theorem). This led us to consider the absolute weak$^*$ topology $|\sigma|(E^{*},E)$ on the dual $E^{*}$ of a Banach lattice, which is obviously defined as the locally convex-solid topology on $E^{*}$ generated by the family $\{q_{x}: x \in E\}$ of lattice seminorms, where $q_x(x^{*}) = |x^{*}|(|x|)$ for all $x^{*} \in E^{*}$ and $x \in E$. In particular, we have that
$$\sigma(E^{*}, E) \subset |\sigma|(E^{*},E) \subset |\sigma|(E^{*}, E^{**}).$$

It just so happens that there is no Alaoglu's Theorem for  the absolute weak$^*$ topology:

\begin{proposition}\label{noalaoglu}Let $E$ be a Banach lattice. \label{absweak4}
    If $B_{E^{*}}$ is  absolutely weak$^*$ compact, then $E$ has order continuous norm.
\end{proposition}

\begin{proof}  Since $\sigma(E^{*},E) \subset |\sigma|(E^{*}, E)$, the identity operator
    $$ Id: (B_{E^{*}}, |\sigma|(E^{*},E)) \longrightarrow (B_{E^{*}}, \sigma(E^{*},E) ) $$
    is continuous. As $B_{E^{*}}$ is assumed to be $|\sigma|(E^{*},E)$-compact and the weak$^*$ topology $\sigma(E^{*},E)$ is Hausdorff, we obtain by \cite[Theorem 2.36]{aliprantis} that $Id$ is a homeomorphism. We claim now that $(E^{*}, |\sigma|(E^{*}, E))^{*} = E$. Indeed, on the one hand, given a weak$^*$ continuous linear functional $x^{**}\colon E^{*} \longrightarrow \R$, for every open subset $U$ of  $\R$ we have
    $$(x^{**})^{-1}(U) \in \sigma(E^{*},E) \subset |\sigma|(E^{*},E),$$ showing that $x^{**}$ is $|\sigma|(E^{*},E)$-continuous. On the other hand, if $x^{**} \colon E^{*} \longrightarrow \R$ is a $|\sigma|(E^{*},E)$-continuous linear functional, the restriction of $x^{**}$ to $B_{E^{*}}$ is $|\sigma|(E^{*},E)$-continuous as well. Calling on the inclusion $\sigma(E^{*},E) \subset |\sigma|(E^{*}, E)$ once again, we get that $x^{**}|_{B_{E^{*}}}$ is weak*-continuous. By \cite[Corollary 3.94]{fabian}, $x^{**}$ is weak*-continuous. This proves that
    $$(E^{*}, |\sigma|(E^{*}, E))^{*} = (E^{*}, \sigma(E^{*},E))^{*} = E,$$
where the last equality is well known for Banach spaces. Thus $|\sigma|(E^{*},E)$ is a topology on $E^{*}$ consistent to the dual pair $< E, E^{*} >$, and so $E$ has order continuous norm (see \cite[Theorems 8.60 and 9.22]{aliprantis}).
\end{proof}

Before proceeding, let us see that the Banach lattice $E$ having order continuous norm is not sufficient for $B_{E^{*}}$  to be an absolute weak* compact subset of $E^{*}$:

\begin{example}\label{ex77}\rm Let $1< p < \infty$. The Banach lattice $L_p[0,1]$ has order continuous norm because it is reflexive. Using that the Rademacher sequence $(r_n)_n$ is weakly null in $L_p[0,1]$, the same argument we used in Example \ref{priexe} shows that $(r_n)_n$ has no absolutely weakly null subsequence. So, $B_{L_p[0,1]}$ is not absolutely weakly sequentially compact. 
Theorem \ref{smulian} gives that $B_{L_p[0,1]}$ fails to be absolute weakly compact. Since the absolute weak topology and the absolute weak$^*$ topologies coincide on reflexive Banach lattices, we conclude that $B_{L_p[0,1]}$ is not absolute weak$^*$ compact.
\end{example}

We shall prove a partial converse of Theorem \ref{smulian}, that is, a partial Eberlein part of the Eberlein-Šmulian Theorem. We just need one more tool.

\begin{lemma}\label{absweak5} Let $J \colon E \longrightarrow E^{**}$  be the canonical embedding from a Banach lattice to its bidual. Then
    $$ J \colon (E, |\sigma|(E,E^{*})) \longrightarrow (J(E), |\sigma|(E^{**},E^{*})) $$ is an homeomorphism.
\end{lemma}

\begin{proof} Since $J$ is a linear map between two topological vector spaces, it is enough to prove that $J$ and its inverse $J^{-1} \colon  J(E) \longrightarrow E$ are continuous at zero with respect to the prescribed topologies.  For a net $(x_\alpha)_\alpha$ in $E$, we have
    $$ x_\alpha \cvfa 0 \iff J(|x_\alpha|)(|x^{*}|) = |x^{*}|(|x_\alpha|) \longrightarrow 0 \text{~for every~} x^{*} \in E^{*}. $$
 The fact that the canonical embedding is a lattice homomorphism (see \cite[Proposition 1.4.5]{meyer}), we obtain that
    $$ x_\alpha \cvfa 0 \iff |J(x_\alpha)|(|x^{*}|) \longrightarrow 0 \text{~for every~} x^{*} \in E^{*} \iff J(x_\alpha) \cvfdea 0, $$
which completes the proof. 
\end{proof}

\begin{theorem}\label{eberlein} Let $K$ be an absolutely weakly sequentially compact subset of a Banach lattice $E$. If $E$ is separable or $B_{E^{**}}$ is absolutely weak$^*$ compact, then $K$ is absolutely weakly compact.
\end{theorem}

\begin{proof}  If $E$ is a separable Banach lattice, the result follows from Lemma \ref{absweak2} and the fact that compactness and sequential compactness are equivalent on metrizable spaces. 

Suppose that $B_{E^{**}}$ is a $|\sigma|(E^{**}, E^{*})$-compact set. Let us stress first that $K$ is norm bounded. Indeed, if $K$ were not norm bounded, there would exist a sequence $(x_n)_n$ in $ K$ such that $\|x_n\|  \geq n$ for every $n \in \N$. By the absolute weak sequential compactness of $K$, there would exist a subsequence $(x_{n_k})_k$ of $(x_n)_n$ and $x \in K$ such that $x_{n_k} \cvfa x$, that is, $|x_{n_k}| \stackrel{\omega}{\longrightarrow} |x|$. This implies that $(|x_{n_k}|)_k$ is a bounded sequence, which is a contradiction because $\||x_{n_k}|\|= \|x_{n_k}\| \geq k$ for every $k \in \N$. This establishes that $K$ is norm bounded. Now, if $J \colon E \longrightarrow E^{**}$ denotes the canonical embedding, $J(K)$ is a norm bounded subset of $E^{**}$, so there exists $C > 0$ such that $J(K) \subset C \cdot B_{E^{**}}$. On the other hand, the $|\sigma|(E^{**}, E^{*})$ compactness of $B_{E^{**}}$ gives that $C \cdot B_{E^{**}}$ is $|\sigma|(E^{**}, E^{*})$-compact as well. It follows that $\overline{J(K)}^{|\sigma|(E^{**}, E^{*})}$ is $|\sigma|(E^{**}, E^{*})$-compact.
     Let $x^{**} \in \overline{J(K)}^{|\sigma|(E^{**},E^{*})} $ be given. From the fact that $\sigma(E^{**}, E^{*}) \subset |\sigma|(E^{**}, E^{*})$ we know that $x^{**} \in \overline{J(K)}^{\sigma(E^{**},E^{*})}$. Let $(x_n)_n \subset K$ be the sequence given by Whitley's Lemma (see \cite[Lemma 3.39]{alip}) associated to $x^{**}$. As $K$ is absolutely weakly sequentially compact, there exist a subsequence $(x_{n_k})_k$ of $(x_n)_n$ and $x \in K$ such that $x_{n_k}  \cvfa x$. This implies that $x$ is a weak accumulation point of $\{x_n: n \in \N\}$; therefore, by Whitley's Lemma, $x^{**} = J(x) \in J(K)$. We have just proved that $J(K) = \overline{J(K)}^{|\sigma|(E^{**},E^{*})}$, which gives that $J(K)$ is $|\sigma|(E^{**}, E^{*})$-compact. By Lemma \ref{absweak5} we conclude that $K$ is absolutely weakly compact.
\end{proof}

To establish the usefulness of the theorem above in the nonseparable case, we should give examples of nonseparable Banach lattices $E$ for which  $B_{E^{**}}$ is absolutely weak$^*$ compact.

\begin{proposition} Let $E$ be Banach lattice such that $E^{*}$ and $E^{**}$ have order continuous norms and $E^{**}$ is atomic. Then $B_{E^{**}}$ is absolutely weak$^*$ compact. In particular, $B_{E^{**}}= B_{E}$ is absolutely weak$^*$ compact for every reflexive atomic Banach lattice.
\end{proposition}

\begin{proof} Let $(x^{**}_{\alpha})_\alpha$ be a net in $B_{E^{**}}$. From Alaoglu's Theorem there exist a subnet $(x^{**}_\beta)_\beta$ of $(x^{**}_\alpha)_\alpha$ and an element $x^{**} \in B_{E^{**}}$ such that $x^{**}_\beta \stackrel{\omega^*}{\longrightarrow} x^{**}$ in $E^{**}$. In particular, $(x^{**}_\beta - x^{**}) \stackrel{\omega^*}{\longrightarrow} 0$ in $E^{**}$. The assumptions on $E^{*}$ and $E^{**}$ imply, by \cite[Theorems 8.4 and 8.1]{kandic}, that  
$(x^{**}_\beta - x^{**}) \stackrel{|\sigma|(E^{**},E^{*})}{\longrightarrow} 0$, that is, $x^{**}_\beta   \stackrel{|\sigma|(E^{**},E^{*})}{\longrightarrow} x^{**}$. Therefore every net in $B_{E^{**}}$ has an absolute weak* convergent subnet to an element of $B_{E^{**}},$ which implies that $B_{E^{**}}$ is $|\sigma|(E^{**}, E^{*})$-compact. The second assertion follows from the fact that reflexive Banach lattices have order continuous norms.
\end{proof}

\begin{example}\rm Let $\Gamma$ be an uncountable set and $1 < p < \infty$. Then $\ell_p(\Gamma)$ is a nonseparable reflexive atomic Banach lattice (for the nonseparability, see \cite[p.\,23]{fabian}). By the proposition above,
$$\{\ell_p(\Gamma) : 1 < p < \infty \mbox{~and~} \Gamma \mbox{ is uncountable}\} $$
is a quite large family of nonseparable Banach lattices $E$ for which $B_E = B_{E^{**}}$ is absolutely weak$^*$ compact.
\end{example}

\section{Grothendieck's compactness principle}

As announced in the Introduction, the purpose of this section is to prove that an analogue of Grothendieck's compactness principle for the weak topology in a Banach space holds for the absolute weak topology in a Banach lattice by replacing the Schur property with the positive Schur property. The two main theorems of the previous section shall be used. As an application, we prove a new characterization of the dual positive Schur property by means of the absolute weak$^*$ topology.

Recall that a Banach lattice $E$ has the positive Schur property if positive (or, equivalently, disjoint or positive disjoint) weakly null sequences in $E$ are norm null. This property was introduced by W. Wnuk \cite{wnukglasgow, wnuksurv} and 
F. R\"abiger \cite{rabiger} and has been extensively studied for many experts, recent developments can be found, e.g., in \cite{ardakani, baklouti, botelhobu, pams, botelholuiz,  pedro}.

The following connections between the positive Schur property and the absolute weak topology shall be useful. The proof follows easily from the facts stated in Remark \ref{remconv}.

\begin{proposition} \label{psp1} The following are equivalent for a Banach lattice $E$: \\
{\rm (a)} $E$ has the positive Schur property.\\
{\rm (b)} Absolutely weakly null sequences in $E$ are norm null.\\
{\rm (c)} Positive absolutely weakly null sequences in $E$ are norm null.\\
{\rm (d)} Positive disjoint absolutely weakly null sequences in $E$ are norm null.\\
{\rm (e)} Disjoint absolutely weakly null sequences in $E$ are norm null.
\end{proposition}

Now we need some preparatory lemmas.

\begin{lemma} \label{absweak7}
     For every absolutely weakly null sequence $(x_n)_n$ in a Banach lattice $E$, the set
     $$\left\{\displaystyle \sum_{n=1}^\infty a_n x_n: a_n\in\mathbb{R} \mbox{~for every~} n\in\mathbb{N} \mbox{~and~} \displaystyle \sum_{n=1}^\infty |a_n| \leq 1\right\}$$
     is absolutely weakly sequentially compact, hence norm closed.
\end{lemma}

\begin{proof} Call $K$ the set whose absolute weak sequential compactness we have to prove. Let $(z_k)_{k}$ be a sequence in $K$. 
 For each $k \in \N$ there exists a sequence of real numbers $(a_{k,n})_n$ such that $\sum\limits_{n=1}^\infty |a_{k,n}| \leq 1$ and $z_k = \sum\limits_{n=1}^\infty a_{k,n} x_n$. By a standard diagonal argument (see, e.g., \cite[p.\,325]{arazi}, \cite[p.\,621]{dean} or \cite[p.\,233]{kim}), we can find a subsequence $(k_j)_j$ of $\N$ and a sequence $(a_n)_n$ such that $\lim\limits_{j \to \infty}a_{k_j,n} = a_n$ for every $n \in \N$ and $\sum\limits_{n=1}^\infty |a_n| \leq 1$.  We claim that $z_{k_j} \cvfa z = \sum\limits_{n=1}^\infty a_n x_n$.
    Given $\varepsilon > 0$ and $0 \leq x^* \in E^*$, since $x_n \cvfa 0$, there exists $N_0 \in \N$ such that $x^*(|x_n|) < \varepsilon/ 4 $ for every $n \geq N_0$. On the other hand, as $a_{k_j,n} \stackrel{j}{\longrightarrow} a_n$  for every $n \in \N$, for each $n = 1, \dots N_0$, there exists $j_n \in \N$ such that $|a_{k_j, n} - a_n| < \varepsilon/(2N_0 L)$ for every $j \geq j_n$, where $L = \sup\limits_{n \in \N} x^*(|x_n|)$. Thus, for $j \geq \max \{ j_1, \dots, j_{N_0}\}$, we have
    \begin{align*}
        x^*(|z_{k_j} - z|) & = x^* \left ( \left | \displaystyle \sum_{n=1}^\infty a_{k_j,n} x_n - \displaystyle \sum_{n=1}^\infty a_n x_n \right | \right ) \\
        & \leq \displaystyle \sum_{n=1}^{N_0} |a_{k_j,n} - a_n| \, x^* (|x_n|) +  \displaystyle \sum_{n=N_0 + 1}^{\infty} |a_{k_j,n} - a_n| \, x^* (|x_n|) \\
        & < \displaystyle \sum_{n=1}^{N_0} \frac{\varepsilon}{2N_0 L} x^*(|x_n|) +  \frac{\varepsilon}{4}\left( \displaystyle \sum_{n=N_0 + 1}^\infty |a_{k_j,n}| + \displaystyle \sum_{n=N_0 + 1}^\infty  |a_n|\right)\\
        & < \frac{\varepsilon}{2} + \frac{\varepsilon}{2} = \varepsilon.
    \end{align*}
 This shows that $z_{k_j} - z \cvfa 0$, hence $z_{k_j} \cvfa z$ as claimed. Since $z \in K$, this proves that $K$ is absolutely weakly sequentially compact. From the inclusion $ \sigma(E,E^*) \subset |\sigma|(E,E^*)$ we get that $K$ is weakly sequentially compact. It is easy to see that absolutely weakly sequentially compact sets are norm closed. 
\end{proof}
%

\begin{lemma} \label{convexa}The closed convex hull of any absolutely weakly null sequence in a Banach lattice is
absolutely weakly compact.
\end{lemma}

\begin{proof} Let $x_n \cvfa 0$ in a Banach lattice $E$. The set
$$K:= \left\{\displaystyle \sum_{n=1}^\infty a_n x_n: a_n\in\mathbb{R} \mbox{~for every~} n\in\mathbb{N} \mbox{~and~} \displaystyle \sum_{n=1}^\infty |a_n| \leq 1\right\}$$
is absolutely weakly sequentially compact by Lemma \ref{absweak7}. 
Setting $C : = \text{\rm co} \{x_n : n \in \N\}$ and denoting by $\tau_{\|\cdot\|}$ the norm topology of $E$,  we have
$$ \overline{C}:= \overline{C}^{\tau_{\|\cdot\|}} \subset \overline{C}^{|\sigma|(E,E^*)} \subset \overline{C}^{\sigma(E,E^*)} = \overline{C},
$$
where the first two inclusions follow from $ \sigma(E,E^*) \subset |\sigma|(E,E^*) \subset \tau_{\|\cdot\|}$ and the last one from Mazur's theorem (recall that $C$ is convex). It follows that $\overline{C} = \overline{C}^{|\sigma|(E,E^*)}$, so $\overline{C}$ 
is a $|\sigma|(E,E^*)$-closed subset of $E$. Observe that $K$ is norm closed (Lemma \ref{absweak7}), convex (easy) and contains $\{x_n : n \in \mathbb{N}\}$ (obvious). It follows that $\overline{C} = \overline{\rm co}\{x_n : n \in \mathbb{N}\}\subset K$. In summary, $\overline{C}$ is a $|\sigma|(E,E^*)$-closed subset of $E$ contained in the $|\sigma|(E,E^*)$-sequentially compact set $K$. Therefore, every sequence in $\overline{C}$ has a $|\sigma|(E,E^*)$-convergent subsequence to an element of $\overline{C}$.

Putting $X = \overline{ \text{span}} \{x_n: n \in \N\}$, it follows from \cite[Exercise 9, p.\,204]{alip} that there exists a separable Banach sublattice $F$ of $E$ containing $X$. Let us see that every sequence in $\overline{C}$ has a $|\sigma|(F,F^*)$-convergent subsequence to an element of $\overline{C}$. Indeed, given $(z_n)_n$ in $\overline{C}$, by the first part of the proof there exist a subsequence $(z_{n_k})_k$ of $(z_n)_n$ and $z \in \overline{C}$ such that $z_{n_k} \cvfa z$, and so $|x^*|(|z_{n_k} - z|) \longrightarrow 0$ for every $x^* \in E^*$. For any $y^* \in F^*$, we can extend $|y^*|$ to a positive linear functional $x^* \in E^*$ \cite[Corollary 1.3]{lotz}, hence $|y^*|(|z_{n_k} - z|) =  |x^*|(|z_{n_k} - z|) \longrightarrow 0$, that is, $z_{n_k} \stackrel{|\sigma|(F,F^*)}{\longrightarrow} z \in \overline{C}$. This proves that $\overline{C}$ is $|\sigma|(F,F^*)$-sequentially compact. Since $F$ is separable, $\overline{C}$ is $|\sigma|(F,F^*)$-compact by Theorem \ref{eberlein}.  It follows from  Lemma \ref{induzida} that 
$\overline{C}$ is $|\sigma|(E,E^*)$-compact.
\end{proof}

According to the reasoning above, $\overline{C} : = \overline{C}^{\tau_{\|\cdot\|}} = \overline{C}^{|\sigma|(E,E^*)}$ for every convex subset $C$ of a Banach lattice $E$. For simplicity, we shall henceforth write $\overline{C}$.

The proof of  the next lemma can be found within the proof of \cite[Theorem 1]{dowling}.

\begin{lemma} \label{lema}
    In a Banach space $X$ one cannot find weakly null sequences $(x_n)_n$ and $(y_n)_n$ with $\|x_n\| = 1$ for every $n \in \N$ such that
    $$ {\textstyle\bigcup\limits_{n=1}^\infty} \left [ (n \cdot \overline{\rm co}(\{x_n : n \in \N\})) \cap \frac{1}{n} B_X \right ] \subset \overline{\rm co}(\{y_n: n \in \N\}). $$
\end{lemma}

Actually the fact stated in the lemma above is one of the most difficult parts of the proofs of \cite[Theorem 1]{dowling} and \cite[Theorem 3]{mupasiri}. Fortunately, it can also be used to prove our main result:

\begin{theorem} \label{mainteo} A Banach lattice $E$ has the Schur property if and only if every absolutely weakly compact subset of $E$ is contained in the closed convex hull of an absolutely weakly null sequence.
\end{theorem}

\begin{proof} Assume first that $E$ has the positive Schur property. Let $K$ be an absolutely weakly compact subset of $E$ and let $(x_n)_n$ be a sequence in $K$. By Theorem \ref{smulian} there exist a subsequence $(n_k)_k$ of $ \N$ and $x \in K$ such that $x_{n_k} \cvfa x$, that is, $(x_{n_k} - x)_k$ is an absolutely weakly null sequence in $E$. Proposition \ref{psp1} gives that $\|x_{n_k} - x\| \longrightarrow 0$, proving that $K$ is norm compact. It follows from Grothendieck's compactness principle that there exists a norm null sequence $(y_n)_n \subset E$ such that $K \subset \overline{\text{\rm co}}(\{y_n: n \in \N\})$. Since the absolute weak topology is contained in the norm topology we have that $(y_n)_n $ is absolutely weakly null.  

     Now assume that every $|\sigma|(E, E^*)$-compact subset of $E$ is contained in the $|\sigma|(E, E^*)$-closed convex hull of a $|\sigma|(E, E^*)$-null sequence.
    Suppose that $E$ fails the positive Schur property. Then there exists, by Proposition \ref{psp1}, a $|\sigma|(E, E^*)$-null sequence $(z_n)  \subset E$ which is not norm null. Without loss of generality, we may assume that there exists $\varepsilon > 0$ such that $\|z_n\| \geq \varepsilon$ for every $n \in \N$. Take $x_n = \frac{z_n }{ \|z_n\|}$ for each $n \in \N$. Using that $\frac{1}{\|z_n\|} \leq \frac{1}{\varepsilon}$ for every $n$, it follows that $(x_n)_n$ is a $|\sigma|(E, E^*)$-null sequence with $\|x_n\| = 1$ for every $n \in \N$. Define
$$ K_n = (n \cdot \overline{\text{\rm co}}(\{x_n: n \in \N\})) \cap \frac{1}{n} B_E \mbox{ for each } n \mbox{~~and~~}  K = \textstyle\bigcup\limits_{n=1}^\infty K_n. $$
Each $K_n$ is a closed convex subset of $E$, hence $K_n$ is a $|\sigma|(E,E^*)$-closed subset of the $|\sigma|(E,E^*)$-compact set $n \cdot \overline{\text{\rm co}}(\{x_n: n \in \N\})$ (Lemma \ref{convexa}). So each $K_n$ is $|\sigma|(E,E^*)$-compact. We claim that $K$ is $|\sigma|(E, E^*)$-compact. Indeed, let $\mathcal{C}$ be a $|\sigma|(E, E^*)$-open cover of $K$. Since $x_n \in K_1$ for every $n$, $(x_n)_n$ is  a $|\sigma|(E,E^*)$-null sequence contained in $K$, so $0 \in \overline{K}^{|\sigma|(E, E^*)}$. Thus we can  take $U \in \mathcal{C}$ so that $0 \in U$. As $U$ is $|\sigma|(E, E^*)$-open and contains $0$, it is contained in a basic $|\sigma|(E, E^*)$-neighborhood of $0$, that is, there exist $\delta > 0$ and $x_1^*, \dots, x_k^* \in (E^*)^+$ such that
$$ V(\delta, x_1^*, \dots, x_k^*) = \{x \in E: x_i^*(|x|) < \delta, \, \forall i = 1, \dots, k\} \subset U. $$
Fix $0< \eta < \frac{\delta}{\max_{1 \leq i \leq k} \|x_i^*\|}$ and let $N_0 \in \N$ be such that $\frac{1}{j} < \eta$ for every $j > N_0$. We have
$$ K_j \subset \eta B_E \subset V(\delta, x_1^*, \dots, x_k^*) \subset U$$
for every $j > N_0$. As a finite union of $|\sigma|(E, E^*)$-compact sets,   $K_1 \cup \cdots \cup K_{N_0}$ is $|\sigma|(E, E^*)$-compact itself, so we can pick $U_1, \dots, U_{N_1} \in \mathcal{C}$ such that
$ K_1 \cup \cdots \cup K_{N_0} \subset U_1 \cup \dots \cup U_{N_1}, $ which yields that
$$ K \subset U_1 \cup \cdots \cup U_{N_1} \cup U. $$
This proves that $K$ is $|\sigma|(E, E^*)$-compact. By assumption, there exists a $|\sigma|(E, E^*)$-null sequence $(y_n)_n \subset E$ such that $K \subset \overline{\text{\rm co}}(\{y_n: n \in \N\}).$ Since $(x_n)_n$ and $(y_n)_n$ are, in particular, weakly null sequences in $E$, Lemma \ref{lema} yields a contradiction that completes the proof.
\end{proof}

It is natural to wonder how the dual positive Schur property can be connected to the absolute weak topology. Recall from \cite{wnukdual} that a Banach lattice $E$ has the dual positive Schur property (respectively, the positive Grothendieck property) if every positive weak* null sequence in $E^*$ is norm null (respectively, weakly null). Both of these properties were introduced by Wnuk in \cite{wnukdual} and they have attracted
the attention of many experts; for recent developments we refer the reader to \cite{bu, galindomiranda, mach, pedro}. It was pointed out by Wnuk that $E$ has the dual positive Schur property if and only if $E$ has the positive Grothendieck property and $E^*$  has the positive Schur property. Combining this characterization with Theorem \ref{mainteo} and Theorem \ref{smulian}, we obtain the following connections between the absolute weak$^*$ topology and the dual positive Schur property.

\begin{corollary} \label{corolario}
 The following are equivalent for a Banach lattice $E$:\\
{\rm (a)} $E$ has the dual positive Schur property.\\
{\rm (b)} Every absolutely weak* null sequence in $E^*$ is norm null.\\
{\rm (c)} $E$ has the positive Grothendieck property and every sequentially absolutely weak$^*$-compact subset of $E^*$ is contained in the closed convex hull of an absolutely weak$^*$ null sequence.
\end{corollary}

\begin{proof} (a) $\Rightarrow$ (b)  If $(x_n^*)_n$ is an absolute weak* null sequence in $E^*$, then $(|x_n^*|)_n$ is a weak* null sequence in $E^*$, so $\|x_n^*\| = \||x_n^*|\| \longrightarrow 0$.

\medskip

\noindent (b) $\Rightarrow$ (c) If $(x_n^*)_n$ is a positive weak* null sequence in $E^*$, then $(x_n^*)_n$ is absolute weak* null, so $\|x_n^*\| \longrightarrow 0$ in $E^*$. In particular, $x_n^* \stackrel{\omega}{\longrightarrow} 0$, proving that $E$ has the positive Grothendieck property. Now, let $K $ be a $|\sigma|(E^*,E)$- sequentially compact subset of $E^*$. Given $(x_n^*)_n \subset K$, take $(x_{n_k})_k$ such that $x_{n_k}^* \cvfea x^*$, that is, $(x_{n_k}^* - x^*)_k$ is a $|\sigma|(E^*,E)$-null sequence in $E^*$. By assumption, $\|x_{n_k}^* - x^*\| \longrightarrow 0$, so $K$ is norm compact. By Grothendieck's compactness principle, there exists a norm null sequence $(y_n^*)_n$ in $ E^*$ such that $K \subset \overline{\rm co}(\{y_n^*: n \in \N\})$. We are done because $y_n^* \cvfea 0$.

\medskip

\noindent (c) $\Rightarrow$ (a) Since $E$ has the positive Grothendieck property, by \cite{wnukdual} all that is left to be proved is that $E^*$ has the positive Schur property. To do so, let $K$ is an absolutely weakly compact subset of $E^*$. Then $K$ is absolutely weakly sequentially compact by Theorem \ref{smulian}, so $K$ is absolutely weak$^*$ sequentially compact. By assumption, there exists an absolute weak$^*$ null sequence $(y_n^*)_n$ in $E^*$ such that
    $ K \subset \overline{\rm co}(\{y_n^*: n \in \N\}).$
    From $y_n^* \cvfea 0$ in $E^*$ we get that $|y_n^*| \cvfe 0$ in $E^*$. Since $E$ has the positive Grothendieck property, $|y_n^*| \cvf 0$, so $y_n^* \cvfa 0$ in $E^*$. By Theorem \ref{mainteo} we conclude that $E^*$ has the positive Schur property.
\end{proof}

\bigskip

\noindent G. Botelho and V. C. C. Miranda\\
Faculdade de Matem\'atica\\
Universidade Federal de Uberl\^andia\\
38.400-902 -- Uberl\^andia -- Brazil\\
e-mails: botelho@ufu.br, colferaiv@gmail.com 

\medskip

\noindent J. L. P. Luiz\\
Instituto Federal do Norte de Minas Gerais\\
Campus de Ara\c cua\'i\\
39.600-00 -- Ara\c cua\'i -- Brazil\\
e-mail: lucasvt09@hotmail.com

%
%
%

\end{document}